\documentclass[11pt,a4paper]{amsart}
\usepackage{amscd,amssymb,amsopn,amsmath,amsthm,graphics,amsfonts,enumerate,verbatim,calc}
\usepackage[dvips]{graphicx}
\input xy
\xyoption{all}

\usepackage{amssymb,amsmath}

\headsep=1cm \numberwithin{equation}{section}
\newtheorem{theorem}{Theorem}[section]
\newtheorem{lemma}[theorem]{Lemma}

\newtheorem{corollary}[theorem]{Corollary}

\newtheorem{notation}[theorem]{Notation}
\theoremstyle{definition}
\newtheorem{definition}[theorem]{Definition}
\theoremstyle{remark}
\newtheorem{remark}[theorem]{Remark}
\newtheorem{example}[theorem]{Example}
\newtheorem{question}[theorem]{Question}

\newtheorem{acknowledgement}{Acknowledgement}

\newcommand{\Min}{\operatorname{Min}}

\newcommand{\sym}{\operatorname{Sym}}

\newcommand{\Kgrade}{\operatorname{K.grade}}

\newcommand{\wAss}{\operatorname{wAss}}
\newcommand{\Spec}{\operatorname{Spec}}

\newcommand{\Ht}{\operatorname{ht}}

\newcommand{\V}{\operatorname{V}}
\newcommand{\Var}{\operatorname{Var}}

\newcommand{\Hom}{\operatorname{Hom}}

\newcommand{\fp}{\frak{p}}
\newcommand{\fq}{\frak{q}}
\newcommand{\fa}{\frak{a}}
\newcommand{\fb}{\frak{b}}

\begin{document}

\author[Asgharzadeh, Dorreh and Tousi]{Mohsen Asgharzadeh, Mehdi Dorreh and Massoud Tousi}

\title[ Direct summands of infinite-dimensional polynomial rings]
{Direct summands of infinite-dimensional polynomial rings}

\address{M. Asgharzadeh, School of Mathematics, Institute for Research in Fundamental
Sciences (IPM), P. O. Box 19395-5746, Tehran, Iran.}
\email{asgharzadeh@ipm.ir}
\address{M. Dorreh,  Department of Mathematics,  Shahid Beheshti
University, G.C., Tehran, Iran. }
\email{mdorreh@ipm.ir}
\address{M. Tousi, Department of Mathematics, Shahid Beheshti
University, G.C., Tehran, Iran, and   School of Mathematics,  Institute for Research in Fundamental
Sciences (IPM), P. O. Box 19395-5746, Tehran, Iran.}

\email{mtousi@ipm.ir} 

\keywords{ Cohen-Macaulay ring; direct summand;  non-Noetherian ring; polynomial ring; purity.\\The first author   was  supported by a grant from IPM, no. 91130407.\\
The third author   was  supported by a grant from IPM, no. 91130211.
}

\begin{abstract}
Let $k$  be a field and $R$  a pure subring  of the infinite-dimensional polynomial ring $k[X_1,\ldots]$. If $R$ is generated by monomials, then  we show that the equality
of  height and grade holds for all ideals of $R$. Also, we show $R$ satisfies the weak
Bourbaki unmixed property.
As an application,
we give  the Cohen-Macaulay property of the invariant ring of  the action of a linearly reductive group acting by $k$-automorphism    on $k[X_1,\ldots]$. This provides several examples of non-Noetherian  Cohen-Macaulay rings (e.g.  Veronese, determinantal and Grassmanian rings).
\end{abstract}

\maketitle

\section{Introduction}

In this paper we are interested in  the following property
of finite-dimensional polynomial rings which is a version of Hochster-Roberts theorem (see
\cite{HR}):

\begin{theorem} Let $S:=k[X_1,\ldots,X_n]$ be a polynomial ring over a field $k$, and let $R$ be
an $\mathbb{N}$-graded subring of $S$ which is pure in $S$. Then $R$ is Cohen-Macaulay.
\end{theorem}

The historical reason for this interest comes from the Cohen-Macaulayness of the invariant ring of  the action of  linearly reductive groups on  polynomial rings. For more details; see \cite[Theorem 6.5.1]{BH}. Suppose a ring $R$ is pure in a Noetherian regular ring which contains a field. As a result of the existence of balanced big
Cohen-Macaulay algebras, $R$ is Cohen-Macaulay, see \cite[Theorem 2.3]{HH}.

Recently, the notion of Cohen-Macaulayness generalized to the non-Noetherian situation; see \cite{HM} and \cite{AT}.
One difficulty is the failure of several classical ideal theory results such as the principal ideal theorem.
In absence of these ideal theory results, the relationship between
dimension theory and homological algebra are given by the following two samples.
Denote the \textit{Koszul} grade by $\Kgrade$. The first easy sample is the following inequality $$\Kgrade_R (\fa,R)\leq \Ht_{R}(\fa),$$which was proved in \cite[Lemma 3.2]{AT}. If the equality is achieved for all ideals, we say $R$ is Cohen-Macaulay in the sense of ideals. The second sample
is the  $\check{C}$\textit{ech} cohomology that used by Hamilton and Marley to define the notion of
strong parameter sequence. A ring $R$ is called
Cohen-Macaulay in the sense of \textit{Hamilton-Marley} if each strong
parameter sequence on $R$ is a regular sequence on $R$.
For more details, see Definition \ref{def1}.

Theorem 1.1 can be extended in two different directions.
First, focus on non-Noetherian finite-dimensional subrings of $k[X_1,\ldots,X_n]$. This kind of investigation was initiated in
\cite{AD} and \cite{ADT}. Second, focus on the infinite-dimensional version of Theorem 1.1.

\begin{notation}
By $R[X_1,\ldots]$, we mean $\bigcup_{i=1}^{\infty}R[X_1,\ldots,X_i]$.
\end{notation}

We refer the reader to \cite{AH}, \cite{HS} and their references to see some properties of
infinite-dimensional polynomial rings via  algebraic statistics and chemistry motivations.

In this paper we attempt to obtain, mostly by a direct limit argument,
results on the widely unknown realm of the infinite-dimensional ring $k[X_1,\ldots]$.  More explicitly,  we are interested in the following question.

\begin{question}\label{Q}
Suppose that $k$ is a field and $R$ is a pure subring of $S:=k[X_1,\ldots]$.   Is $R$ Cohen-Macaulay?
\end{question}

More generally, let $\mathcal{P}$ be a property of commutative Noetherian rings. There is a \textit{cut-paste} idea  to extend
this property to the realm of non-Noetherian rings. To explain the idea, let  $R$ be a non-Noetherian
ring. We refer $\mathcal{P}$ as the \textit{cut} property when $R$ is  written as a direct limit of Noetherian rings satisfying $\mathcal{P}$. If the property  $\mathcal{P}$ behaves  nicely with the direct limit, we refer it as the \textit{paste} property.
We  apply a cut-paste idea
to give a positive answer to Question \ref{Q}
 when $$R \cap k[X_1,\ldots,X_n] \hookrightarrow k[X_1,\ldots,X_n]$$ is pure
for sufficiently large $n$; see  Theorem \ref{htpoly}. It is worth noting that this condition holds, when $R$ is generated by monomials (see  Corollary \ref{htpoly1}). For an application, recall that a linear algebraic group over $k$ is called linearly reductive if every
$G$-module $V$ is a direct sum of irreducible $G$-submodules. For more details; see Remark \ref{red}.  Then
Theorem \ref{htpoly} can be restated as follows.

 \begin{corollary} (see  Corollary \ref{inv})
Let $k$  be an algebraically closed field and $A=k[X_1,\ldots]$. Suppose $G$ is a linearly reductive group over $k$ acting on $A$
(in the sense of Remark 3.7(ii)) by a degree preserving action.
Then $A^G$ is
Cohen-Macaulay in the sense of each part of
Definition \ref{def1}.
\end{corollary}

Veronese, determinantal and the  Grassmanian  rings are important sources of Cohen-Macaulay
rings (see \cite{BV}). They are subrings of a finite-dimensional polynomial ring over a field.
One can extend  their definitions to the case of an infinite-dimensional polynomial ring.
 We study their  Cohen-Macaulayness  in Section 4.

\section{Preliminary Lemmas}

This section  contains five Lemmata. They do not involve any Cohen-
Macaulay concept. We will use them in the next section.

\begin{lemma}\label{min}
 Let $k$  be a field and $I$ a finitely generated ideal of $S = k[X_1,\ldots]$.  Then each minimal prime ideal of $I$ is finitely generated.
\end{lemma}

\begin{proof}
Let $\{f_1,\ldots,f_n\}$ be a generating set for $I$ and $\fp \in \min_S(I)$. Take $m$ be such that $f_i \in R:=k[X_1,\ldots,X_m]$ for all $1\leq i \leq n$. Observe that
 $ (\fp \cap R) S$  is  prime, because   $$R[X_{m+1},\ldots]/\fq R[X_{m+1},\ldots] \cong R/\fq[X_{m+1},\ldots],$$ for all $\fq \in \Spec(R)$. Also, $ I = (I \cap R)S$. In view of $$ I = (I \cap R)S \subseteq  (\fp \cap R) S  \subseteq \fp,$$
we see that $\fp = (\fp \cap R) S$. Clearly, $\fp \cap R$ is finitely generated as an ideal of $R$. So, $\fp$ is a finitely generated ideal of $S$.
\end{proof}

Let $I$ be an ideal of  a ring $R$. By $\Var_R(I)$ we mean the set of all prime ideals of $R$ containing $I$.
Also, $\min_R(I)$ denoted the set of all minimal prime ideals of $I$.

\begin{lemma}\label {fg}
Let $ R \to S$ be a pure ring homomorphism and $I$ an ideal of $R$. Let $\fp \in \min_R(I)$. Then there exists $\fq \in \min_S(IS)$ such that $\fq \cap R = \fp$. In particular,   if $\min_S(IS)$ is finite, then $\min_R(I)$ is finite.
\end{lemma}

\begin{proof}
Since $ IS \cap R = I$, we have a natural injective homomorphism $R/I \hookrightarrow S/IS$. Note that  $ \fp \in \min_R(I)$. By \cite[Page 41, Ex. 1]{K}, there exists $\fq \prime \in \Var_S(IS)$ such that $\fq \prime \cap R = \fp$. Let $\fq \in \min_S(IS)$ such that $ \fq \subseteq \fq \prime$. Then $$ I \subseteq \fq \cap R \subseteq {\fq \prime} \cap R = \fp.$$
So $\fq \cap R = \fp $. This completes the proof.
\end{proof}

\begin{lemma}\label{fg1}
Let $R \hookrightarrow  S := k[X_1,\ldots]$ be a pure ring homomorphism  and $I$ a finitely generated ideal of $R$. Then $\min_R(I)$ is finite.
\end{lemma}

\begin{proof}
In view of  Lemma \ref{min}, elements of $\min_S(IS)$ are finitely generated. By \cite[Theorem]{An},  $\min_S(IS)$ is finite.
The claim  now follows by Lemma \ref{fg}.
\end{proof}

\begin{lemma}\label{fg2}
Let $k$  be a field, $R \hookrightarrow  k[X_1,\ldots]$ a pure ring homomorphism and   $I$  an ideal of $R$. If $0 \leq n \leq \Ht_R(I)$, then there are $x_1,\ldots,x_n \in I$ such that $$i \leq \Ht_R((x_1,\ldots,x_i)R)$$ for all $0 \leq i \leq n $.
\end{lemma}

\begin{proof}
We prove the claim by induction on $n$.  The first step of induction is obvious. Now, suppose $n > 0$ and the claim  has been proved for all $j < n$.  Suppose $ j < n$. By inductive hypothesis, we can find  $x_1,\ldots, x_j \in I$ such that $$i \leq \Ht_R((x_1,\ldots,x_i)R)$$ for all $1\leq i \leq j$.
Set $$X:= \{Q \in \min((x_1,\ldots,x_j)R): \Ht_R(Q)=j\}.$$  Suppose  $X=\emptyset$.
Then, $j + 1 \leq \Ht_R((x_1,\ldots,x_j)R)$. Hence, we have $$j+1 \leq \Ht_R((x_1,\ldots,x_j,x_{j+1})R)$$  for all $x_{j+1} \in I$. Thus, without loss of the generality, we may and do assume that  $X\neq\emptyset$. It follows by Lemma \ref{fg1} that $X$ is finite. Note that $ I \nsubseteqq \bigcup_{Q \in X} Q$, unless $\Ht_R(I) \leq j < n$ which is impossible by our assumptions. Pick $x_{j+1} \in I \setminus \bigcup_{Q \in X} Q$. Thus, $$ j \leq \Ht_R((x_1,\ldots,x_j)R) \leq \Ht_R((x_1,\ldots,x_{j+1})R).$$ But, $\Ht_R((x_1,\ldots,x_{j+1})R) \neq j$. So  $j+1 \leq \Ht_R((x_1,\ldots,x_{j+1})R)$.
\end{proof}

The module case of the next result (when the base ring is fixed) is well-known.
\begin{lemma}\label{pure}
Let $R$, $S$ and $T$ be commutative rings. Let $\varphi : R \to S$ and $\theta: S \to T$ be ring homomorphisms. The following hold:
\begin{enumerate}
\item[$\mathrm{(i)}$] If $\varphi$ and $\theta$ are pure, then $\theta \varphi$ is pure.
\item[$\mathrm{(ii)}$] If $\theta \varphi$ is pure, then $\varphi$ is pure.
\end{enumerate}
\end{lemma}

\begin{proof}
Let $M$ be an $R$-module. Set $\psi:= \theta \otimes Id_{S \otimes_R M}$. Then the following diagram is commutative:
$$
 \xymatrix{ R \otimes_{R} M  \ar[r]^{\varphi \otimes {Id_M}} &  S \otimes_{R} M \ar[r]^{\theta \otimes {Id_M}} \ar[d]  & T \otimes_{R} M   \ar[d]\\
 & (S \otimes_ {S} S) \otimes _{R} M \ar[r] \ar[d] & (T \otimes_ {S} S) \otimes _{R} M \ar[d]\\
 & S \otimes _{S} (S \otimes_{R} M) \ar[r]^{\psi} &  T \otimes _{S} (S \otimes_{R} M),}
$$where columns are isomorphism.
Now we prove the lemma.
\begin{enumerate}
\item[$\mathrm{(i)}$] If $\varphi$ and $\theta$ are pure, then ${\varphi \otimes {Id_M}}$ and $\psi$ are one-to-one. So ${\varphi \otimes {Id_M}}$ and ${\theta \otimes {Id_M}}$ are one-to-one. It is now clear that $\theta \varphi$ is pure, because $({\theta \otimes {Id_M}})({\varphi \otimes {Id_M}}) = {\theta \varphi \otimes {Id_M}}$.
\item[$\mathrm{(ii)}$] If $\theta \varphi$ is pure, then $({\theta \otimes {Id_M}})({\varphi \otimes {Id_M}}) = {\theta \varphi \otimes {Id_M}}$ is one-to-one. Hence ${\varphi \otimes {Id_M}} $ is one-to-one. Therefore $\varphi$ is pure.
\end{enumerate}
\end{proof}

\section{infinite-dimensional Cohen-Macaulay rings}

Our main result in this section is Theorem \ref{htpoly} and its corollaries.
Let $\fa$ be an ideal of a ring $R$ and $M$ an
$R$-module. Suppose first that $\fa$ is
finitely generated  with the  generating set
$\underline{x}:=x_{1},\ldots, x_{r}$. Denote the Koszul
complex  of $R$ with respect to $\underline{x}$ by
$\mathbb{K}_{\bullet}(\underline{x})$. Koszul grade of $\fa$ on $M$
is defined by
$$\Kgrade_R(\fa,M):=\inf\{i \in\mathbb{N}\cup\{0\} | H^{i}(\Hom_R(
\mathbb{K}_{\bullet}(\underline{x}), M)) \neq0\}.$$ Note that by
\cite[Corollary 1.6.22]{BH} and \cite[Proposition 1.6.10 (d)]{BH},
this does not depend on the choice of generating sets of $\fa$.
Suppose now that $\fa$ is a general ideal (not necessarily finitely generated). Take $\Sigma$ to be the family of all finitely generated
subideals $\fb$ of $\fa$.
The Koszul grade of
$\fa$ on $M$ can be defined by
$$\Kgrade_R(\fa,M):=\sup\{\Kgrade_R(\fb,M):\fb\in\Sigma\}.$$ By using
\cite[Proposition 9.1.2 (f)]{BH}, this definition coincides with the
original definition for finitely generated ideals.

\begin{remark} (i): A system $\underline{x} = x_{1}, \ldots,
x_\ell$ of elements of $R$ is called a weak regular sequence on $M$ if $x_i$ is a nonzero-divisor on $M/(x_1,\ldots, x_{i-1})M$ for all $i=1,\ldots,\ell$. The classical grade  of an ideal $\fa$ on $M$ is defined to be the supremum of the lengths of all weak regular sequences on $M$ contained in $\fa$.

(ii): Recall that the classical grade coincides with the Koszul grade if the
ring and the module both are Noetherian.

(iii): Let $R$ be a ring, $M$ an $R$-module and  $\underline{x} = x_{1}, \ldots,
x_\ell$ a sequence of elements of $R$. For each $m \geq n$, there is a chain map $\varphi^{m} _{n} (\underline{x})
:\mathbb{K}_{\bullet}(\underline{x}^{m})\longrightarrow
\mathbb{K}_{\bullet}(\underline{x}^{n}),$ which is induced via
multiplication by $(\prod x_{i})^{m-n}$.
Recall from \cite{Sch} that $\underline{x}$ is
weak proregular if for each $n>0$ there exists an $m \geq n$ such
that the maps $H_{i}(\varphi^{m} _{n} (\underline{x})) :H_{i}
(\mathbb{K}_{\bullet}(\underline{x}^{m}))\longrightarrow H_{i}
(\mathbb{K}_{\bullet}(\underline{x}^{n}))$ are zero for all $i \geq
1$.
\end{remark}

Now, we recall the following key definitions:

\begin{definition}\label{def1} (See \cite[Definition 3.1]{AT} and its references.) Let $R$ be a ring. \begin{enumerate}
\item[$\mathrm{(i)}$] $R$ is called
Cohen-Macaulay in the sense of \textit{Glaz} if for each prime ideal $\fp$
of $R$, $$\Ht_{R}(\fp)=\Kgrade_{R_{\fp}}(\fp R_{\fp},R_{\fp}).$$
\item[$\mathrm{(ii)}$]   Recall that a prime ideal $\fp$ is \textit{weakly associated} to a module $M$ if $\fp$
is minimal over $(0 :_R m)$ for some $m\in M$. We denote the set of
weakly associated primes of $M$ by $\wAss_RM$. Let $\fa$ be a finitely
generated ideal of $R$. Set $\mu(\fa)$ for the minimal number of
elements of $R$ that needs to generate $\fa$.
Assume that for each
ideal $\fa$ with the property $\Ht (\fa)\geq \mu(\fa)$, we have $\min
(\fa)=\wAss_R( R/\fa)$. A ring with such a property is called \textit{weak
Bourbaki unmixed} (Abb. by WB). For more details see \cite{Ha}.

\item[$\mathrm{(iii)}$]
By  $H^{i}_{\underline{x}}(M)$, we mean the $i$-th cohomology of the \v{C}ech complex of $M$ with respect to $\underline{x}$.
Adopt the above notation. Then
$\underline{x}$ is called a \textit{parameter sequence} on
$R$, if:
(1) $\underline{x}$ is a weak proregular sequence; (2)
$(\underline{x})R \neq R$; and (3) $H^{\ell}_{\underline{x}}
(R)_{\fp} \neq 0$ for all $\fp \in
\V(\underline{x}R)$. Also, $\underline{x}$ is called a strong
parameter sequence on $R$ if $x_{1},\ldots, x_{i}$ is a parameter
sequence on $R$ for all $1\leq i \leq \ell$. $R$ is called
Cohen-Macaulay in the sense of \textit{Hamilton-Marley} (Abb. by HM) if each strong
parameter sequence on $R$ is a regular sequence on $R$. For more details see \cite{HM}.

\item[$\mathrm{(iv)}$] Let $\mathcal{A}$ be a non-empty  class of
ideals of a ring $R$. $R$ is called Cohen-Macaulay in the sense
of $\mathcal{A}$, if $\Ht_{R}(\fa)=\Kgrade_R(\fa,R)$ for all
$\fa\in\mathcal{A}$. We denote this property by $\mathcal{A}$.
The classes that we are interested in, are $\Spec(R)$,
$\max (R)$, the class of all ideals and the class of
all finitely generated ideals.
 \end{enumerate}
\end{definition}

\begin{remark} \label{rem}The following diagram was proved in \cite[3.2. Relations]{AT}:\[\begin{array}{lllllll}
\textmd{Max}\Leftarrow\textmd{Spec}
\Leftrightarrow\textmd{ideals}\Rightarrow\textmd{Glaz}
\Rightarrow\textmd{f.g.
ideals}\Rightarrow\textmd{HM}\Leftarrow\textmd{WB}.
\end{array}\]
Also, when  the base ring is coherent,
$\textmd{Spec}\Rightarrow\textmd{WB}$.
\end{remark}

The following will play an essential role  in the proof of Corollary \ref{htpoly1}.

\begin{theorem}\label{htpoly}
Let $k$  be a field and $R$   a subring of $k[X_1,\ldots]$ containing $k$. Assume that there is a strictly increasing infinite sequence $\{b_n\}_{n\in \mathbb{N}}$ of positive integers
such that $R \cap k[X_1,\ldots,X_{b_n}] \hookrightarrow k[X_1,\ldots,X_{b_n}]$ is pure
 for all $n \in \mathbb{N}$. Then $R$ is Cohen-Macaulay in the sense of each part of
Definition \ref{def1}.
\end{theorem}

\begin{proof}
We denote $R \cap k[X_1,\ldots,X_{b_n}]$ by $R_n$ for all $n \in \mathbb{N}$.  In view of Remark \ref{rem}, we need to show that $R$ is Cohen-Macaulay in the sense of ideals
and $R$ is weak
Bourbaki unmixed.
 Let $n \in \mathbb{N}$. One has $R_n \cap (Jk[X_1,\ldots,X_{b_n}]) = J$ for every ideal $J$ of $R_n$. So $R_n$ is a Noetherian ring. Therefore, we deduce from \cite[Theorem 10.4.1 and Remark 10.4.2]{BH} that $R_n$ is a Noetherian Cohen-Macaulay ring. Keep in mind that the ring homomorphism $k[X_i:1\leq i\leq b_n]\to k[X_i:1\leq i<\infty]$ is pure.
By looking at the following commutative diagram and Lemma \ref{pure}
$$
\xymatrix{R_{n}\ar[r] \ar[d] & k[X_i:1\leq i\leq b_n] \ar[d]\\
 R \ar[r] & k[X_i:1\leq i<\infty],}
$$
the ring homomorphism $R_n \to R$ is pure.
Also, by \cite[Lemma 3.9]{ADT}, $$R = \bigcup_{n\in \mathbb{N}}R_n \to k[X_i:1\leq i<\infty] = \bigcup_{n\in \mathbb{N}}k[X_1,\ldots,X_{b_n}]$$ is a pure ring homomorphism.

$\mathrm{(i)}$ First we show that $R$ is Cohen-Macaulay in the sense of ideals. Let $I$ be an ideal of $R$ such that $n \leq \Ht_R(I)$. We use  Lemma \ref{fg2}, to find elements $a_1,\ldots,a_n \in I$ such that $$i \leq \Ht_R((a_1,\ldots,a_i)R), $$ for all $1\leq i \leq n$.

Now we claim  that, for each $1\leq i \leq n$, there exists $l_i \in \mathbb{N}$ such that $a_1/1, \ldots, a_i/1$ is a regular sequence in $(R_k)_{\fq}$ for every $k\geq l_i$ and $\fq \in \Var_{R_{k}}((a_1,\ldots,a_i)R_{k})$.
To this end, let $1 \leq i \leq n$. In view of Lemma \ref{fg1}, $\min_R(a_1,\ldots,a_i)R$ is finite. Denote it by $ \{Q_1,\ldots,Q_m\} $. We have  the following chain of prime ideals$$P_{j_0} \subsetneqq \ldots \subsetneqq P_{j_{i}}=Q_j$$for all  $1 \leq j \leq m$. Pick $b_{j_{t}} \in P_{j_{t}}\setminus P_{j_{t-1}}$ for all $1 \leq j \leq m$ and $1\leq t\leq i$. Set $$Y:= \{b_{jt}|   1 \leq j \leq m , 1 \leq t \leq i \}.$$ Since $Y$ is finite, there exists $\ell_i \in \mathbb{N}$ such that $ Y\subseteq R_{\ell_i} $ and  $\{a_1,\ldots,a_n \}\subseteq R_{\ell_i}$. We use this to deduce   that $$i \leq \Ht_{R_{k}}( Q_j \cap R_k)  $$ for all  $1 \leq j \leq m$ and $\ell_i \leq k$.
Let $\ell_i \leq k$. By Lemma \ref{fg}, for each $ \fp \in \min_{R_{k}}(( a_1, \ldots, a_i)R_k)$, there is $1\leq j \leq m$ such that $Q_j \cap R_k = \fp$. Hence $\Ht_{R_{k}}((a_1,\ldots,a_i)R_{k}) \geq i.$  The reverse inequality  holds, because $R_{k}$ is Noetherian. So $$\Ht_{(R_{k})_{\fq}}((a_1,\ldots,a_i)(R_{k})_{\fq})) = i$$ for all  $\fq \in \Var_{R_{k}}((a_1,\ldots,a_i)R_{k})$. Since $(R_k)_\fq$ is a Noetherian Cohen-Macaulay local ring,  $a_1/1, \ldots, a_i/1$ is a regular sequence in $(R_k)_\fq$. This proves the claim.

Set $l := \max \{\ell_1,\ldots,\ell_n\}$ and fix $k\geq l$. Then $a_1/1, \ldots, a_i/1$ is a regular sequence in $(R_k)_\fq$ for all $1 \leq i \leq n$ and $\fq \in \Var_{R_{k}}((a_1,\ldots,a_i)R_{k})$. Then $a_1, \ldots,a_n$ is a regular sequence in $R_k$  for all $k\geq l$. Hence $a_1, \ldots,a_n$ is a weak regular sequence in $R$. Consequently $n \leq \Kgrade_R(I,R)$. So $\Ht_R(I) \leq \Kgrade_R(I,R)$. The reverse inequality is always true by \cite[Lemma 3.2]{AT}.

$\mathrm{(ii)}$ Here we show that $R$ is weak
Bourbaki unmixed.
 Let $\fa $ be a proper finitely generated ideal of $R$ with the property that $\Ht(\fa) \geq \mu(\fa)$.
Set $\ell:=\mu(\fa)$ and let $\underline{y}:= y_1,\ldots,y_{\ell}$ be a generating set for $\fa$.   In view of Lemma \ref{fg2}, there exists $\underline{x}:= x_1,\ldots,x_l$ in $\fa$ such that $$i \leq \Ht_R((x_1,\ldots,x_i)R),$$ for each $1\leq i \leq \ell $.
Let $1\leq i\leq \ell$ and set $\fa_i:=(x_1,\ldots,x_i)R$. In view of part $\mathrm{(i)}$,  $$i\leq\Ht \fa_i=\Kgrade_R(\fa_i,R) \leq  \mu(\fa_i)\leq i.$$
So by \cite[Proposition 3.3(e)]{HM}, $\underline{x}$ is a strong parameter sequence on $R$. In view of Remark \ref{rem}, $R$ is Cohen-Macaulay in the sense of Hamilton-Marley. Therefore, $\underline{x}$ is a regular sequence on $R$.

   There are $r_{ij} \in R$ such that $x_i = \sum _{1 \leq j \leq l }r_{ij}y_j$ for all $1 \leq i \leq l$. Recall that $R_m = R \cap k[X_1,\ldots,X_{b_m}]$ for all $m$. Take  $n \in \mathbb{N}$ be such that all of $r_{ij}$,  $\underline{x}$ and $\underline{y}$  belong to $R_m$ for all $m \geq n$.

 Suppose $\fp \in \wAss(R/{\fa})$. Clearly $\underline{x}$ is a regular sequence on $R_\fp$. Set $\fp_m := \fp \cap R_m$ for all $m \geq n$. The purity of $R_m \to R$ implies that $\underline{x}$ is  a regular sequence on $R_m$ (see \cite [Proposition 6.4.4]{BH}). Then $\underline{x}$ is a regular sequence on ${R_m}_{(\fp_m)}$. Note that  $\underline{x}R_m\subseteq\underline{y}R_m$. Thus $$\ell \leq \Ht_{{R_m}_{(\fp_m)}}(\underline{x}{R_m}_{(\fp_m)}) \leq \Ht_{{R_m}_{(\fp_m)}}(\underline{y}{R_m}_{(\fp_m)})\leq \ell.$$ Since ${R_m}_{(\fp_m)}$ is a Noetherian Cohen-Macaulay local ring, we see $\underline{y}$ is a regular sequence on ${R_m}_{(\fp_m)}$. Thus $\underline{y}$ is a regular sequence on $R_{\fp}$ .

In view of \cite[Theorem 3.3]{AT} and  \cite[Lemma 3.5]{AT}, $R_{\fp}/{\underline{y}R_{\fp}}$ is Cohen-Macaulay in the sense of ideals. It follows from   \cite[Lemma 3.9]{AT} that $\wAss_{R_{\fp}}( R_{\fp}/{\underline{y}R_{\fp}}) = \Min ( R_{\fp}/{\underline{y}R_{\fp}})$ and so $\fp \in \Min(\fa)$.
\end{proof}

\begin{remark}
$\mathrm{(i)}$
As Remark  \ref{rem} says, Cohen-Macaulay in the sense of ideals implies weak
Bourbaki unmixedness,
when  the base ring is coherent.
Note that in the above theorem  $R$ is not
necessarily coherent (see \cite[Example 2]{G2}).

$\mathrm{(ii)}$ It may be worth to note that one can construct a direct system of Noetherian
Cohen-Macaulay rings such that its  direct limit is
not Cohen-Macaulay, see \cite[Example 4.7]{ADT}.
\end{remark}

We are now ready to prove:

\begin{corollary}\label{htpoly1}
Let $k$ be a field and $R$ a pure $k$-subalgebra of $S = k[X_1, \ldots]$  generated by monomials.  Then $R$ is
Cohen-Macaulay in the sense of each part of
Definition \ref{def1}.
\end{corollary}

\begin{proof}
There is a  natural projection $\pi_n:k[X_1,\ldots]\to k[X_1,\ldots,X_n]$ defined by evaluation: for each $f \in k[X_1,\ldots]$, $\pi_n(f)$ is given by the substitution $X_{n+i}= 0$  $\forall i \geq 1$. Set $R_n:= R\cap k[X_1,\ldots,X_n]$ for all $n\in \mathbb{N}$.

  We claim that $$\pi_n(R)\subseteq R_n.$$ To see this, let $r\in R$. As $R$ is generated by monomials,  $r=a_0+\ldots+a_m$
where $a_i \in R$ is a monomial. Note that either $\pi_n(a_i)=0$ or $\pi_n(a_i)=a_i$. In both cases
$\pi_n(r)\in R$, as claimed.

  Hence, we can define $\pi_n:R\to R_n$ and $\pi_n$ provides a retraction for the natural inclusion $R_n \to R$. Thus $R_n$ is a direct summand of $R$ as an $R_n$-module for all $n \in \mathbb{N}$.

   Let $n\in \mathbb{N}$. It follows from the following commutative diagram
   $$
\xymatrix{R_n\ar[r] \ar[d] & k[X_1,\ldots,X_n] \ar[d]\\
 R \ar[r] & k[X_1,\ldots],}
$$
that the ring homomorphism $R_n \hookrightarrow k[X_1,\ldots,X_n]$ is pure. Now,  the claim is an  immediate consequence of  Theorem \ref{htpoly}.
\end{proof}

In the following we cite some aspects of invariant theory that we need in the sequel. We refer the reader to \cite{HH1} and \cite{BH} for
more details.

\begin{remark}\label{red}Let $k$ be an algebraically closed field.
\begin{enumerate}
\item[(i)]
 Recall that a \textit{linear algebraic group} over $k$ is a Zariski closed subgroup of some $GL(V):=Aut_k(V)$, where $V$ is a finite-dimensional $k$-vector space. By a homomorphism of linear algebraic groups we mean a group homomorphism which is a morphism of varieties.
\item[(ii)]
Let $G$ be a linear algebraic group. Then $G$ acts $k$-\textit{rationally} on a finite-dimensional $k$-vector space $V$ if the map $\Phi : G \to GL(V)$ defining the action is a homomorphism of the linear algebraic groups. If $V$ is infinite-dimensional, $G$ acts  $k$-rationally on $V$ if the action is such that $V$ is a union of finite-dimensional $G$-stable subspaces $W$ such that $G$ acts $k$-rationally on $W$ in the sense above.

If $R$ is a $k$-algebra, $G$ acts on $R$ to mean that $G$ acts $k$-rationally on the $k$-vector space $R$ by $k$-algebra automorphism.  In invariant theory, it is commonly accepted that
"an action of an algebraic group on a $k$-algebra" means a rational one.  So, in the sequel we treat only with rational actions on $k$-algebras.

When $G$ acts $k$-rationally on a $k$-vector space $V$, we shall say that $V$ is a $G$-\textit{module}. Recall that  $U\subset V$ is said to be $G$-submodule, if it is a vector subspace of $V$ and $g(u)\in U$ for all $g \in G$ and $u \in U$.
 Also, $U$ is called \textit{irreducible} if it has no nontrivial $G$-submodule.

\item[(iii)]Let $G$ be a  linear algebraic group. Then $G$ is called
\textit{linearly reductive}, if every $G$-module $V$ is a direct sum of irreducible $G$-submodules. An equivalent condition is that every $G$-submodule  $W$ of  $V$ has a $G$-stable complement $L$, i.e., $V=W\oplus L$ as $G$-modules (see \cite[Page 170]{HH1}).

The most classical examples of linearly reductive groups are finite groups $G$ whose order is not divisible by $char$  $k$. In characteristic $0$, the groups $GL(n,k)$ and $SL(n,k)$ are linearly reductive, and so are the orthogonal and sympletic groups. The tori $GL(1,k)^m$ are linearly reductive independently of $char$ $k$ (see \cite[Page 292]{BH}).
\item[(iv)]
Let $G$ be a linearly reductive group and $V$ be a $G$-module. Let $V^G$ be the subspace of invariants, i.e., $$V^G = \{ v \in V : \; for \; all \; g \in G, \; g(v) = v \}.$$ Then $V^G$ is the largest $G$-submodule of $V$ on which $G$ acts trivially. Let $W$ be the sum of all irreducible $G$-subspaces of $V$ on which $G$ acts non-trivially. Then $V   = V^G \oplus W$, and $W$ is the unique complementary $G$-subspace of $V$ (see \cite[Page 170]{HH1}).
\item[(v)]
Let $G$ be a linearly reductive group and $R = k[X_1,\ldots, X_n]$.  Denote the graded component  containing homogenous  elements of degree $i$ of $R$ by $R_i$. Suppose $G$ acts   on $R$ by degree-preserving  $k$-algebra homomorphisms. This means that $g(R_n) \subseteq R_n$ for all $g \in G$ and $n \in \mathbb{N}$. Then $$R^G = \{f \in R : g(f)=f \; for \; all \; g \in G\}$$ is the ring of invariants.  There exists a finite dimensional representation $\varphi _i : G \to GL_k(R_i)$ for each $i$. By $(iv)$, $R_i = R_i^G \oplus W_i$ for each $i$. Then $$R = \oplus_{i \geq 0} R_i \cong ( \oplus_{i \geq 0} R_i^G) \oplus (\oplus_{i\geq 0} W_i).$$ Keep in mind  that the action is degree preserving. Then we have $R^G = ( \oplus_{i \geq 0} R_i^G)$. Set $W = (\oplus_{i\geq 0} W_i)$. We show that $W$ is an $R^G$-module. Consequently, $R^G$ is a direct summand of $R$ as an $R^G$-module.

Let $r \in R^G$ and $a \in W$. Then $r = r_1 + \ldots +r_t$ and $a = a_1 + \ldots +a_t$ where $r_i \in R_i^G$ and $a_i \in W_i$. For each $a_j$, there exists an irreducible $G$-subspace $U$ of $W_j$ such that $a_j \in U$. Consider the $G$-homomorphism $r_i : U \to r_i U$. This map is zero or one-to-one. If the map is zero, then  $r_i U=0\subseteq  W_{i+j}$. If the map is one-to-one, then $U\simeq r_i U$ as $G$-spaces. It follows that
$G$ acts nontrivially on the irreducible $G$-space  $r_i U$. Since $r_i U\subseteq R_{i+j}$, one has $r_i U\subseteq  W_{i+j}$. So, $ra\in W$
and $W$ is an $R^G$-module.
\end{enumerate}
\end{remark}

Now we are ready to prove the following result.

\begin{corollary}\label{inv}
Let $k$  be an algebraically closed field and $A=k[X_1,\ldots]$. Suppose $G$ is a linearly reductive group over $k$
acting on $A$  by  degree-preserving $k$-algebra automorphisms.
Then $A^G$ is
Cohen-Macaulay in the sense of each part of
Definition \ref{def1}.
\end{corollary}

\begin{proof}
Indeed, for simplicity, assume that each $X_i$ is of degree one. Set
$V:=\bigoplus_{i=1}^{\infty} kX_i$. It is easy to see that $V$ is a $G$-module.
Then by Remark 3.7 $\mathrm{(iii)}$, there is a decomposition $ V = \bigoplus V_i$ with each $V_i$ a finite-dimensional $G$-submodule of $V$. Now set $b_0 = 0$,
and $b_i =
\sum_{j=1}^i \dim _k V_j$, and take a $k$-basis $\{Y_{b_{i-1}+1},\ldots, Y_{b_i}\}$
of $V_i$ for each
$i \geq 1$.
The notation $\sym_{k}(W)$ stands for the \textit{symmetric} algebra of a $k$-vector space $W$.
Recall  from \cite[8.3.3 and 8.3.5]{G1} the following two items:\\

\begin{enumerate}
\item[$\mathrm{(i)}$] $\sym_k(V)=\sym(\bigoplus V_i)  \simeq\bigcup \sym_k(V_i)=\bigcup k[Y_{b_{i-1}+1},\ldots, Y_{b_i}]$,
\item[$\mathrm{(ii)}$] $\sym_k(V)=\sym(\bigoplus_n (\bigoplus_{i=1}^{n} kX_i) ) \simeq\bigcup \sym_k(\bigoplus_{i=1}^{n} kX_i)=\bigcup k[X_1,\ldots,X_n]$.
\end{enumerate}

Then, without loss of the generality one can replace $\{X_1,\ldots\}$ by the new variables $\{Y_i\}$. That is, there is a strictly increasing infinite sequence $\{b_n\}_{n\in \mathbb{N}}$ of positive integers such that
$$k[Y_1,\ldots,Y_{b_n}] \  \ is \  \ a \;\textit{G-submodule}\; of \; A \; \; \;  \forall n \in \mathbb{N}.$$
For each $n \in \mathbb{N}$, set $A_n:= k[Y_1,\ldots,Y_{b_n}]$. Then $G$ acts on $A_n$ by degree-preserving $k$-algebra automorphisms.
 By Remark \ref{red} (v), $A_n^G$ is a direct summand of $A_n$ as an $A_n^G$-module. Hence  $A_n^G\to A_n$  is pure. By applying Theorem \ref{htpoly}, $A^G$ is Cohen-Macaulay in the sense of each part of
Definition \ref{def1}.
\end{proof}

Also, Question \ref{Q} has an affirmative answer in the following case:

\begin{remark}\label{att}
Let $R = k[x_1, \ldots]$ be an infinite-dimensional polynomial ring over a field $k$ and $G$ a finite
group of automorphisms of $R$ such that the order of $G$ is a unit
in $R$.  Recall from \cite[Theorem 4.1]{AT} that $R$ is Cohen-Macaulay in the sense of each part of
Definition \ref{def1}.
By \cite[Theorem 5.6 and Proposition 5.7]{AT}, $R^G$ is
Cohen-Macaulay in the sense of each part of
Definition \ref{def1}.
\end{remark}

In the next section we give several examples in context of corollary \ref{inv}. As a special case, the next result  provides more evidence for
an affirmative answer for Question \ref{Q}.

\begin{example}\label{ee}
Let $k$ be a field with $\mathrm{char }(k)\neq2$ and $S = k[X_1, \ldots]$. The assignments
$X_{2i+1}\mapsto X_{2i+2}$ and $X_{2i}\mapsto X_{2i-1}$ define an automorphism
$g:S\to S$. Let $G$ be the group  generated by $g$. Then \begin{enumerate}
\item[$\mathrm{(i)}$]  The ring $k[X_1,\ldots,X_n]$ is not $G$-submodule of $S$ for all $n \in \mathbb{N}$.
\item[$\mathrm{(ii)}$]  The ring  $R:=S^G$ can not be generated by monomials.
\item[$\mathrm{(iii)}$] The ring $R$  is
Cohen-Macaulay in the sense of each part of
Definition \ref{def1}.
\end{enumerate}
\end{example}
\begin{proof}
Note that the order of $g$ is two. So, $G=\{1,g\}$.
\begin{enumerate}
\item[$\mathrm{(i)}$] This is trivial.
\item[$\mathrm{(ii)}$] It is clear that $X_1+X_{2}$ is invariant by $G$. If $R$ were generated by monomials,
then  $X_1$ and $X_{2}$ should be invariant, which is impossible.
\item[$\mathrm{(iii)}$]  The order of $G$ is invertible in $S$ and $S$ is
Cohen-Macaulay in the sense of each part of
Definition \ref{def1}. To conclude the argument  see Remark \ref{att}.
\end{enumerate}
\end{proof}

\section{Examples}
Next, we present several examples of non-Noetherian  Cohen-Macaulay rings, as an application of our main result.
The following gives Cohen-Macaulayness of infinite-dimensional determinantal rings.

\begin{example}
Let $\{z_{ij}:i,j\in\mathbb{N}\}$ be a family of variables
over an algebraically closed field $k$ of characteristic $0$. Let $Z:=(z_{ij})$ be a matrix. We denote the polynomial ring $k[z_{ij}:i,j\in\mathbb{N}]$ by $k[Z]$.
Let  $I_{n}(Z)$ be the ideal of  $k[Z]$  generated by the $n$-minors of $Z$.
Then $k[Z] / I_{n+1}(Z)$ is Cohen-Macaulay in the sense of each part of
Definition \ref{def1}.
\end{example}

\begin{proof}
First note that by an $n$-minor of $Z$ we mean  the determinant of an $n\times n$ submatrix of $Z$. Let $\{x_{ij}:i\in\mathbb{N}, 1\leq j\leq n\}$ and $\{y_{jk}:k\in\mathbb{N}, 1\leq j\leq n\}$ be two families of variables
over  $k$. Define the matrices
$X := (x_{ij})$ and $Y := (y_{jk})$. Look at the polynomial ring $R = k[X,Y]$. First, we show that $k[XY]\cong k[Z] / I_{n+1}(Z)$.

Consider the matrices $X_m = (x_{ij})_{1\leq i\leq m , 1 \leq j \leq n}$ and $Y_m = (y_{jk})_{ 1 \leq j \leq n,1\leq k\leq m}$ where $m$ is an integer greater than $n+1$. Let $Z_m = (z_{ij})_{ 1 \leq i \leq m,1\leq j\leq m}$ be an $m \times m$ submatrix of $Z$.  Then there exists the homomorphism of $k$-algebras $\varphi_m: k[Z_m] / I_{n+1}(Z_m) \to k[X_m,Y_m]$ such that $Z_m + I_{n+1}(Z_m) \to X_mY_m$. By \cite[Theorem 7.2]{BV}, $\varphi_m$ is an embedding. So the induced homomorphism $\hat{\varphi_m}: k[Z_m] / I_{n+1}(Z_m) \to k[X_mY_m]$ is an isomorphism. For each $m,l$ such that $n+1 \leq m \leq l$, let $$\pi_{ml}:  k[Z_m] / I_{n+1}(Z_m) \longrightarrow k[Z_l] / I_{n+1}(Z_l)\; \; \; and \; \; \;  \; \;  \lambda_{ml} : k[X_mY_m] \longrightarrow k[X_lY_l]$$ be the natural homomorphism of $k$-algebras. Then $$\{\hat{\varphi_m}\}_{m \geq n+1} : (k[Z_m] / I_{n+1}(Z_m), \pi_{ml}) \longrightarrow (k[X_mY_m],\lambda_{ml})$$ is an isomorphism of direct systems. On the other hand, $${\varinjlim}_{m \geq n+1} k[Z_m] / I_{n+1}(Z_m) \cong k[Z] / I_{n+1}(Z)\; \; \; and \; \; \;  {\varinjlim}_{m \geq n+1}k[X_mY_m] = k[XY].$$ Hence $k[XY]\cong k[Z] / I_{n+1}(Z)$.

Let $G := GL_n(k)$ be the general linear group. By Remark \ref{red} (iii),  $G $ is linearly reductive. For $M \in G$ and a polynomial $f(X,Y) \in k[X,Y]$ one puts $$M(f) := f(XM^{-1}, MY).$$  As  $M$ runs through $G$, this defines an action of $G$ on $R := k[X,Y]$ as a group of $k$-algebra automorphisms. Denote the polynomial ring $ k[X_m,Y_m]$ by $R_m$  for all $m \in \mathbb{N}$. Then $G$ acts on $R_m$ likewise $R$, i.e. $R_m$ is $G$-stable. By Corollary \ref{inv}, $R^G$ is  Cohen-Macaulay in the sense of ideals.   In order to show $k[Z] / I_{n+1}(Z)$ is Cohen-Macaulay, it is enough to show that $R^G = k[XY]$. In the light of \cite[Proposition 7.4 and Theorem 7.6]{BV}, $R_m^G = k[X_mY_m] $. Also we have $R^G = \cup_{m\in \mathbb{N}} R_m^G$ and $k[XY]= \cup_{m\in \mathbb{N}}k[X_mY_m]$. Therefore $R^G = k[XY]$.
\end{proof}

The following gives Cohen-Macaulayness of infinite-dimensional Grassmanian  rings.

\begin{example}
Let $\{x_{ij}:j\in\mathbb{N}, 1\leq i\leq m\}$  be a family of variables
over an algebraically closed field $k$ of characteristic $0$ and let $X:=(x_{ij})$ be the corresponding matrix. Set $R := k[X]$. Let $Gr_{m\infty}(k)$ be the $k$-subalgebra of $R$ generated by the $m$-minors of $X$. Then $Gr_{m\infty}(k)$ is Cohen-Macaulay in the sense of each part of
Definition \ref{def1}.
\end{example}

\begin{proof}
For each $n \in \mathbb{N}$,  set $X_n := \{x_{ij}: 1 \leq j \leq n, 1\leq i\leq m\}$ and $R_n := k[X_n]$. Suppose $n\geq m$ and  denote the  $k$-subalgebra of $R_n$ generated by the $m$-minors of $X_n$ by $Gr_{mn}(k)$. Clearly,  $Gr_{m\infty}(k) = \cup_{n\geq m}Gr_{mn}(k)$.

Let $G := SL_m(k)$. By Remark \ref{red} (iii),  $G $ is linearly reductive.  $G$ acts on $R$ via the assignment  $X \mapsto TX$ for all $T \in G$. Also, $G$ acts on $R_n$ likewise $R$ for all $n \in \mathbb{N}$.  By \cite[Corollary 7.7]{BV}, $Gr_{mn}(k)= R_n^G$. So $$Gr_{m\infty}(k) = \cup_{n\geq m}Gr_{mn}(k) = \cup_{n\geq m}R_n^G = R^G.$$ Now, it  follows from Corollary \ref{inv} that $Gr_{m\infty}(k)$ is Cohen-Macaulay in the sense of each part of
Definition \ref{def1}.
\end{proof}

The following extends \cite[Corollary 5.8]{AT} to a more general situation.

\begin{example}\label{br2}
Let $k$  be a field and $A:=k[X_1,\ldots]$. We recall the definition of  Veronese rings.
Let $f:=X_{i_1}^{j_1}\ldots X_{i_{\ell}}^{j_{\ell}}$ be a monomial in $A$. The degree of $f$ is defined by $d(f):=\sum_{k=1}^{\ell} j_k$. Let $d$ be a positive integer. We call the $k$-algebra  $A^{(d)}$, generated by all monomials of degree $d$, the $d$-th Veronese subring of $A$. Then $A^{(d)}$   is Cohen-Macaulay in the sense of each part of
Definition \ref{def1}.
\end{example}

\begin{proof}
Denote the Veronese subring  of $A_n=k[X_1,\ldots,X_n]$ by $A^{(d)}_n$. Recall that $A^{(d)}_n$ is  the $k$-subspace of $A_n$ is generated by $$\{X_1^{v_1}\ldots X_n^{v_n}| v_1, \ldots,v_n \in \mathbb{N}_0 , v_1+ \ldots + v_n \equiv 0 \ \ (\textbf{mod } d)\}.$$ Define $\rho: A_n \to A^{(d)}_n$ such that $\rho$ maps each monomial $r \in A_n\setminus A^{(d)}_n$ to $0$ and each  monomial $r \in A^{(d)}_n$ to itself. Extend $\rho$ linearly to $A_n$. One can see easily that $\rho$ is a retraction of $A^{(d)}_n$ to $A_n$. So, $A^{(d)}_n$ is a  direct summand of $A_n$.  It turns out that the ring extension  $ A^{(d)}_n\to A_n$ is pure.   On the other hand $A^{(d)} \cap A_n = A^{(d)}_n$. By applying Theorem \ref{htpoly}, $A^{(d)}$ is Cohen-Macaulay in the sense of each part of
Definition \ref{def1}.
\end{proof}

\begin{example}\label{se}
 Here, we give a natural extension of Example \ref{br2}.
Let $\{X_j : j \in \mathbb{N}\}$  be   a family of variables over a field $k$ and $A:=k[X_1,\ldots]$. Fix $s,t \in \mathbb{N}$ and  choose integers $k_{1,j},\ldots,k_{s,j} \in \mathbb{Z}$ for each $j \in \mathbb{N}$.  Let $H$ be the submonoid of $\mathbb{N}^\infty:=\cup_{n \in \mathbb{N}} \mathbb{N}^n$  consisting of the solutions of the homogeneous linear equations  $$\sum_{1\leq j \leq n} k_{i,j} X_j = 0, 1 \leq i \leq s $$ for all $n \geq t$. Then $H$ is a full subsemigroup of $\mathbb{N}^\infty$ that is for each $\alpha,\beta \in H$ with $\alpha-\beta \in \mathbb{N}^\infty$, one has $\alpha-\beta \in H$. Let $W$ be the $k$-span of the monomials $X_1^{a_1} \ldots X_n^{a_n}$ such that $(a_1,\ldots, a_n,0,\ldots) \in \mathbb{N}^\infty \setminus H$. If $\beta \in \mathbb{N}^\infty \setminus H$ and $\alpha \in H$, then $\alpha + \beta \in \mathbb{N}^\infty \setminus H$. Hence, $W$ is a $k[H]$-module   and $k[H]$ is direct summand of $A$. Since $k[H]$ is  a  $k$-subalgebra of $A$, is generated by monomials, then by  Corollary \ref{htpoly1}, $k[H]$ is Cohen-Macaulay in the sense of each part of
Definition \ref{def1}.
\end{example}

\begin{remark}
 In view of \cite {H}, the ring $k[H]$ of  Example \ref{se} appears in the following way.
Let $G = GL(1,k)^s$ and $\gamma = (\gamma_1,\ldots,\gamma_s) \in G$. The assignments $X_j \mapsto \gamma_1^{k_{1,j}}\ldots\gamma_s^{k_{s,j}} X_j$ define an action of $G$ on $A$.  For any monomial $\lambda = X_1^{a_1}\ldots X_n^{a_n}$ and for each $\gamma = (\gamma_1,\ldots,\gamma_s) \in G$, $\gamma$ sends $\lambda$ to $(\prod_{1\leq i \leq s}(\gamma_i^{k_{i,1}a_1+\ldots+k_{i,n}a_n}))\lambda$. It is well-known that the ring of invariants is spanned over $k$ by all monomials $x_1^{a_1} \ldots  x_n^{a_n}$, where $t \leq n$ and  the equations $$\sum_{1\leq j \leq n} k_{i,j} X_j = 0, , 1 \leq i \leq s$$are solved by $(a_1, \ldots,a_n)$. This means that $A^G = k[H]$.
\end{remark}

\begin{acknowledgement}
  We thanks to the referee who his/her suggestions leads to an improvement in presentation of the manuscript.
\end{acknowledgement}


\end{document}